\documentclass{amsart}
\usepackage{amssymb}
\usepackage{latexsym}
\usepackage[T1]{fontenc}
\usepackage{lmodern}
\usepackage{enumerate}
\usepackage{mathabx}
\usepackage{csquotes}
\usepackage[mathscr]{euscript}
\usepackage{thmtools}
\usepackage{thm-restate}

\makeatletter
\newlength\knuthian@fdfive
\def\mathpal@save#1{\let\was@math@style=#1\relax}
\def\utilde#1{\mathpalette\mathpal@save
              {\setbox124=\hbox{$\was@math@style#1$}%
\setbox125=\hbox{$\fam=3\global\knuthian@fdfive=\fontdimen5\font$}
\setbox125=\hbox{$\widetilde{\vrule height 0pt depth 0pt width \wd124}$}%
               \baselineskip=1pt\relax
               \lineskiplimit=\z@\relax
               \lineskip=1pt\relax
               \vtop{\copy124\copy125\vskip -\knuthian@fdfive}}}
\makeatother

\usepackage{pifont}

\declaretheorem[numberwithin=section]{theorem}

\newtheorem{corollary}[theorem]{Corollary}
\newtheorem{proposition}[theorem]{Proposition}

\newtheorem*{claim}{Claim}

\theoremstyle{definition}
\newtheorem{definition}[theorem]{Definition}

\theoremstyle{remark}
\newtheorem{remark}[theorem]{Remark}

\newcommand{\HJ}{\mathcal{O}}
\newcommand{\ATR}{\mathsf{ATR}}
\newcommand{\DS}{\mathsf{DS}}
\newcommand{\ACA}{\mathsf{ACA}}
\newcommand{\RCA}{\mathsf{RCA}}
\newcommand{\M}{\mathfrak{M}}
\newcommand{\N}{\mathfrak{N}}

\newcommand{\Lang}{\mathcal{L}}

\DeclareMathOperator{\rank}{rank}

\begin{document}

\title{Incompleteness and jump hierarchies}

\author{Patrick Lutz}
\address{Department of Mathematics, University of California, Berkeley}
\email{pglutz@berkeley.edu}

\author{James Walsh}
\address{Group in Logic and the Methodology of Science, University of California, Berkeley}
\email{walsh@math.berkeley.edu}

\subjclass[2010]{Primary 03F35, 03D55; Secondary 03F40}

\commby{Heike Mildenberger}

\thanks{Thanks to Antonio Montalb\'an, Ted Slaman, and an anonymous referee for their comments and suggestions. Also, thanks to Jun Le Goh for alerting us to errors in the statements of theorem \ref{ms2} and corollary \ref{ms3}.}

\begin{abstract}
This paper is an investigation of the relationship between G\"odel's second incompleteness theorem and the well-foundedness of jump hierarchies. It follows from a classic theorem of Spector's that the relation $\{(A,B) \in \mathbb{R}^2 : \mathcal{O}^A \leq_H B\}$ is well-founded. We provide an alternative proof of this fact that uses G\"odel's second incompleteness theorem instead of the theory of admissible ordinals. We then use this result to derive a semantic version of the second incompleteness theorem, originally due to Mummert and Simpson. Finally, we turn to the calculation of the ranks of reals in this well-founded relation. We prove that, for any $A\in\mathbb{R}$, if the rank of $A$ is $\alpha$, then $\omega_1^A$ is the $(1 + \alpha)^{\text{th}}$ admissible ordinal. It follows, assuming suitable large cardinal hypotheses, that, on a cone, the rank of $X$ is $\omega_1^X$.
\end{abstract}

\maketitle

\section{Introduction}

% \blfootnote{2010 \emph{Mathematics Subject Classification.} Primary 03F35, 03D55; Secondary 03F40.}

In this paper we explore a connection between G\"odel's second incompleteness theorem and recursion-theoretic jump hierarchies. Our primary technical contribution is a method for proving the well-foundedness of jump hierarchies; this method crucially involves the second incompleteness theorem. We use this technique to provide a proof of the following theorem:

\begin{restatable}{theorem}{main}
\label{thm-main}
There is no sequence $(A_n)_{n<\omega}$ of reals such that, for each $n$, the hyperjump of $A_{n+1}$ is hyperarithmetical in $A_n$.
\end{restatable}

This theorem is an immediate consequence of a result of Spector's, namely that if $\HJ^A \leq_H B$ then $\omega_1^A < \omega_1^B$ (so the existence of such a sequence $(A_n)_{n<\omega}$ would imply the existence of a descending sequence $\omega_1^{A_0} > \omega_1^{A_1} > \ldots$ in the ordinals). We provide an alternative proof that makes no mention of admissible ordinals, and which has the additional benefit of showing the theorem is provable in $\ACA_0$.

Here is a brief sketch of how our alternative proof works: Consider the theory $\ACA_0 + \DS$ where $\DS$ is a sentence asserting the existence of a sequence of reals as described in Theorem \ref{thm-main}. We work \emph{inside} the theory and let $A_0, A_1, \ldots$ be such a sequence. $\ACA_0$ proves that if the hyperjump of a real exists then there is a $\beta$-model (a model that is correct for $\Sigma^1_1$ sentences) containing it. In this case $\HJ^{A_1}$ exists so there is a $\beta$-model containing $A_1$. Moreover, since all $A_n$'s for $n \geq 1$ are hyperarithmetical in $A_1$, the $\beta$-model will contain all of them. All $\beta$-models are models of $\ACA_0$ (in fact, $\ATR_0$) so it appears this model is a model of the theory $\ACA_0 + \DS$, meaning that the theory proves its own consistency. By G\"odel's second incompleteness theorem, this implies that $\ACA_0$ proves $\neg \DS$.

There is one problem, however. Just because the model contains all the elements of the sequence $(A_n)_{n \geq 1}$ does not mean it contains the sequence itself (here we are thinking of the sequence as a single real whose slices are the $A_n$'s). Indeed, the sequence itself could be much more complicated than any single real in the sequence. In our proof, we overcome this flaw by showing that if there is a descending sequence then there is a descending sequence that is relatively simple---in fact there is one that is hyperarithmetic relative to $A_1$. This means the $\beta$-model above really does contain a descending sequence.

In \cite{friedman1976uniformly}, H. Friedman uses similar ideas to prove the following theorem originally due to Steel:

\begin{restatable}[Steel]{theorem}{Steel}
\label{Steel}
Let $P\subset \mathbb{R}^2$ be arithmetic. Then there is no sequence $(A_n)_{n<\omega}$ such that for every $n$,
\begin{itemize}
\item[(i)] $A_n\geq_T A'_{n+1}$ and 
\item[(ii)] $A_{n+1}$ is the unique $B$ such that $P(A_n,B)$.
\end{itemize}
\end{restatable}

In these proofs we move from the second incompleteness theorem to the well-foundedness (or near well-foundedness) of recursion-theoretic jump hierarchies. In fact, the implication goes in both directions: the well-foundedness of appropriate jump hierarchies entails semantic versions of the second incompleteness theorem. For example, theorem \ref{thm-main} yields a simple and direct proof of the following semantic version of the second incompleteness theorem originally due to Mummert and Simpson (recall that $\Lang_2$ is the standard two-sorted language of second order arithmetic):

\begin{restatable}[Mummert--Simpson]{theorem}{ms}
\label{ms}
Let $T$ be a recursively axiomatized $\Lang_2$ theory. For each $n \geq 1$, if there is a $\beta_n$-model of $T$ then there is a $\beta_n$-model of $T$ which contains no countable coded $\beta_n$-models of $T$.
\end{restatable}

In fact, our proof sharpens the Mummert-Simpson result somewhat by dropping the requirement that $T$ be recursively axiomatized.

A different semantic version of the second incompleteness theorem also follows from theorem \ref{Steel}, as observed by Steel in \cite{steel1975descending}. Namely, the following:

\begin{restatable}[Steel]{theorem}{Steelincompleteness}
\label{Steel2}
Let $T$ be an arithmetically axiomatized $\Lang_2$ theory extending $\ACA_0$. If $T$ has an $\omega$-model then $T$ has an $\omega$-model which contains no countable coded $\omega$-models of $T$.
\end{restatable}

These results all point to a general connection between incompleteness and well-foundedness. Elucidating this connection is the central goal of this paper. Though many of the theorems we prove could also be proved from the application of known methods, we believe that the new techniques are more conducive to achieving our central goal. Additionally, our techniques are able to prove somewhat sharper results than the original methods.

We also investigate directly the well-founded hierarchy at the center of theorem \ref{thm-main}. It follows from that theorem that the relation $A \prec B$ defined by $\HJ^A \leq_H B$ is a well founded partial order. We call the $\prec$ rank of a real its \emph{Spector rank}. There is a recursion-theoretically natural characterization of the Spector ranks of reals:

\begin{restatable}{theorem}{ranks}
\label{thm-ranks}
For any real $A$, the Spector rank of $A$ is $\alpha$ just in case $\omega_1^A$ is the $(1 + \alpha)^{\text{th}}$ admissible ordinal.
\end{restatable}

It follows, assuming suitable large cardinal hypotheses, that, on a cone, the Spector rank of $X$ is $\omega_1^X$. 

Here is our plan for the rest of the paper. In \textsection{\ref{well_foundedness}} we describe related research. In \textsection{\ref{proof}} we prove the main theorem. In \textsection{\ref{incompleteness}} we provide an alternative proof of the Mummert-Simpson theorem. In \textsection{\ref{ranks}} we turn to the calculation of Spector ranks.

\section{Second Incompleteness \& Well-Foundedness}
\label{well_foundedness}

The second incompleteness theorem implies the well-foundedness of various structures (in particular, sequences of models). In turn, the well-foundedness of structures sometimes yields a semantic version of the second incompleteness theorem (in the form of a minimum model theorem). It is worth emphasizing that the former argument does not rely on the theory of transfinite ordinals and the latter argument does not rely on self-reference or fixed point constructions. This point allows us to sharpen certain results. Because we avoid the use of ordinals, we can verify that Theorem \ref{thm-main} is provable in $\ACA_0$; because we avoid self-reference, we can drop the restriction in the statement of Theorem \ref{ms} that $T$ be recursively axiomatized.

We will now describe both types of arguments, describe their historical antecedents, and point to related research.

\subsection{Well-foundedness via incompleteness}

\leavevmode

To derive well-foundedness from incompleteness we work in the theory $T$ + ``there is a descending sequence,'' where $T$ is sound and sufficiently strong. We build a model of $T$ containing a tail of the sequence, yielding a consistency proof of $T$ + ``there is a descending sequence'' \emph{within the theory} $T$ + ``there is a descending sequence.'' By the second incompleteness theorem, this means that $T$ proves that there are no descending sequences.

The main difficulties lie in building a model that is correct enough that if a descending sequence is in the model, the model knows it is descending and in finding a $T$ that is strong enough to prove the model exists but weak enough that the model built satisfies it.

As far as we know, the first arguments of this type are due to H. Friedman. We were inspired, in particular, by H. Friedman's \cite{friedman1976uniformly} proof of a theorem originally due to Steel \cite{steel1975descending}.

\Steel*

Steel's proof is purely recursion-theoretic, whereas Friedman's proof appeals to the second incompleteness theorem. In particular, Friedman supposes that there is an arithmetic counter-example $P$ to Steel's Theorem. He then works in the theory $T:= \RCA +\textrm{``$P$ produces a descending sequence}$'' and uses $P$ to build $\omega$-models of arbitrarily large fragments of $T$. This yields a proof of $\mathsf{Con}(T)$ in $T$, whence $T$ is inconsistent by G\"odel's second incompleteness theorem.

Recently, Pakhomov and the second named author developed proof-theoretic applications of this technique in \cite{pakhomov2018reflection}. They show that there is no sequence $(T_n)_{n<\omega}$ of $\Pi^1_1$ sound extensions of $\ACA_0$ such that, for each $n$, $T_n$ proves the $\Pi^1_1$ soundness of $T_{n+1}$. This result is proved by appeal to the second incompleteness theorem, though it could be proved by showing that a descending sequence $(T_n)_{n<\omega}$ of theories would induce a descending sequence in the ordinals (namely, the associated sequence of proof-theoretic ordinals). They also show that, ``on a cone,'' the rank of a theory in this well-founded ordering coincides with its proof-theoretic ordinal. These results are strikingly similar to the main theorems of this paper.

\subsection{Incompleteness via well-foundedness}

\leavevmode

Here is an informal argument for incompleteness via well-foundedness. Suppose that second incompleteness fails, i.e.\ that a consistent $T$ proves its own consistency. If $T$ also proves the completeness theorem, then every model $\M$ of $T$ has (what it is by the lights of $\M$) a model within it. This produces a nested sequence of models. If these models can be indexed by ordinals, then this produces a descending sequence of ordinals. So the well-foundedness of the ordinals produces some form of the second incompleteness theorem. If we know that the models form a well-founded structure, we can argue directly, without the detour through the ordinals.

To sharpen this argument one must know that the objects that are ``models of $T$'' in the sense of $\M$ are genuinely models of $T$. So one must restrict one's attention to structures that are sufficiently correct. In addition, one must clarify the relation by which the models are being compared and prove that it is well-founded.

An early argument of this sort is attributed to Kuratowski (see \cite{kennedy2015incompleteness, kripke2009collapse}). Set theory cannot prove the following strong form of the consistency of set theory: that there is an $\alpha$ such that $V_\alpha$ is a model of set theory. For if it does then there is $\alpha$ such that $V_\alpha$ is a model of set theory. Since $V_\alpha$ is a model of set theory, there is also a $\beta<\alpha$ such that $V_\beta$ is a model of set theory. Iterating this argument produces an infinite descending sequence of ordinals. Contradiction.

Steel has also developed an argument of this sort. Using his Theorem \ref{Steel}, he demonstrates that if an arithmetically axiomatized theory of second order arithmetic extends $\ACA_0$ and has an $\omega$-model then it has an $\omega$-model which contains no countable coded $\omega$-models of the theory.

\subsection{Kripke structures}

\leavevmode

We conclude this discussion of related work with the following observation. Formalized in the language of modal logic, the statement of G\"odel's second incompleteness theorem characterizes well-founded Kripke frames. Indeed, the formalization corresponds to the least element principle:
$$\diamondsuit\varphi \rightarrow \diamondsuit(\varphi\wedge\neg\diamondsuit\varphi).$$
Its contrapositive (writing $\psi$ for $\neg\varphi$) is a modal formalization of L\"ob's theorem which corresponds to induction\footnote{Note that since we replaced $\varphi$ with $\lnot \varphi$ before taking the contrapositive, the two modal statements are equivalent only as \emph{schemas}.}:
$$\Box (\Box\psi \rightarrow \psi) \rightarrow \Box \psi$$
Beklemishev has suggested that this observation is connected with ordinal analysis. In \cite{beklemishev2004provability}, he uses a modal logic of provability known as $\mathsf{GLP}$ to develop both an ordinal notation system for $\varepsilon_0$ and a novel consistency proof of $\mathsf{PA}$.

\section{The Main Theorem}
\label{proof}

In this section we provide our alternative proof of Theorem \ref{thm-main}.

\subsection{Outline of proof}

\leavevmode

In broad strokes, here is our strategy. We will consider a statement $\DS$ which states that there \emph{is} a descending sequence in the hyperjump hierarchy. We then work in the theory $\ACA_0 + \DS$ and derive the statement $\mathsf{Con}(\ACA_0 + \DS)$. By G\"odel's second incompleteness theorem, this implies that there is a proof of $\neg\DS$ in $\ACA_0$.

To derive $\mathsf{Con}(\ACA_0 + \DS)$ in $\ACA_0 + \DS$, we use the hyperjump of a real to construct a coded $\beta$-model of $\ACA_0$ containing that real. In particular, if we are given a descending sequence then we can use the existence of the hyperjump of the second real in the sequence to find a $\beta$-model containing all the elements of the tail of the sequence. The point is that the tail of a descending sequence is again a descending sequence and $\beta$-models are correct enough to verify this.

The only problem is that while the $\beta$-model we found contains all the elements of the tail it may not contain the tail itself (i.e.\ it may not contain the recursive join of all the elements of the tail). Our strategy to fix this is to show that there is a family of descending sequences which is arithmetically definable relative to some parameter whose hyperjump exists. A $\beta$-model containing this parameter must contain an element of this family (because $\beta$-models contain witnesses to all $\Sigma^1_1$ statements).

For the parameter, we will use a countable coded $\beta$-model which contains each element of a tail of the original descending sequence. The arithmetic formula will then essentially say that the $\beta$-model believes each step along the sequence is descending. The point is that we have replaced a $\Pi^1_1$ formula saying the sequence is descending by an arithmetic formula talking about the truth predicate of some coded model and $\beta$-models are correct enough that this does not cause any errors.

The $\beta$-model will just come from the existence of the hyperjump of some element of the original sequence, and we can guarantee the hyperjump of the model exists by taking one more step down the original descending sequence.

\subsection{Useful facts}

\leavevmode

In this section, we record the facts about $\beta$-models that we will use in the proof of the main theorem. Unless otherwise noted, proofs of all propositions in this section can be found in \cite{simpson_2009}.

\begin{definition}
A $\beta$-model is an $\omega$-model $\M$ of second order arithmetic such that for any $\Sigma^1_1$ sentence $\varphi$ with parameters in $\M$, $\M \vDash \varphi$ if and only if $\varphi$ is true.
\end{definition}

\begin{proposition}[\cite{simpson_2009}, Lemma VII.2.4, Theorem VII.2.7] \label{beta}  Provably in $\ACA_0$, all countable coded $\beta$-models satisfy $\ATR_0$ (and hence also $\ACA_0$).
\end{proposition}

\begin{proposition}[\cite{simpson_2009}, Lemma VII.2.9] \label{hyperjump} Provably in $\ACA_0$, for any $X$, $\HJ^X$ exists if and only if there is a countable coded $\beta$-model containing $X$.
\end{proposition}

\begin{proposition}
\label{absoluteness}
All of the following can be written as Boolean combinations of $\Sigma^1_1$ formulas and hence are absolute between $\beta$-models
\begin{enumerate}
    \item $A$ is the hyperjump of $B$.
    \item $A \leq_H B$
    \item $M$ is a countable coded $\beta$-model.
\end{enumerate}
\end{proposition}

\subsection{Proof of the main theorem}

\leavevmode

\main*

\begin{proof}
It suffices to prove the inconsistency of the theory $\ACA_0+\DS$, where
\[
\DS := \exists X \forall n  (\HJ^{X_{n + 1}} \text{ exists and } \mathcal{O}^{X_{n+1}}\leq_H X_n).
\]
To do this, we reason in $\ACA_0+\DS$ and derive $\mathsf{Con}(\ACA_0+\DS)$. The inconsistency of $\ACA_0+\DS$ then follows from G\"odel's second incompleteness theorem.

\bigskip

\noindent\textbf{Reasoning in $\ACA_0+\DS$:} 

\smallskip

Let $A$ witness $\DS$. That is, for all $n$, $\HJ^{A_{n + 1}}$ exists and $\HJ^{A_{n+1}}\leq_H A_n$. Our goal is now to show there is a model of $\ACA_0 + \DS$.

\begin{claim}
There is a countable coded $\beta$-model $\M$ coded by $M$ such that $\HJ^{M}$ exists and $\M$ contains $A_n$ for all sufficiently large $n$.
\end{claim}

The proof of Proposition \ref{hyperjump} in \cite{simpson_2009} actually shows that for any $X$, if $\HJ^X$ exists then $X$ is contained in a countable coded $\beta$-model which is coded by a real that is recursive in $\HJ^X$. So $A_2$ is contained in some countable coded $\beta$-model $\M$, coded by $M$, such that $M \leq_T \HJ^{A_2} \leq_H A_1$. Hence $\HJ^M \leq_T \HJ^{A_1}$. Since $\HJ^{A_1}$ exists, so does $\HJ^{M}$. And since $\M$ is closed under hyperarithmetic reducibility, $\M$ contains $A_n$ for all $n \geq 2$.

\begin{claim}
There is an arithmetic formula $\varphi$ such that
\begin{enumerate}
    \item[(i)] $\exists X\, \varphi(M, X)$
    \item[(ii)] For any $X$, if $\varphi(M, X)$ holds then $X$ is a witness of $\DS$
\end{enumerate}
where $M$ is as in the previous claim.
\end{claim}

Basically $\varphi(M, X)$ says that $X$ is a sequence of reals whose elements are in $\M$ and for each $n$, $\M$ believes that $\HJ^{X_{n + 1}}$ exists and is hyperarithmetical in $X_n$. More precisely $\varphi(M, X)$ is the sentence
\[
\forall n\, (X_{n + 1}, X_n \in \M \land \M \vDash ``\exists Y\, [Y = \HJ^{X_{n + 1}} \land Y \leq_H X_n]").
\]

To see why $\varphi(M, X)$ has a solution, recall that $\M$ contains $A_n$ for all $n$ sufficiently large. Let $X$ be the sequence $A$ but with the first few elements removed so that $\M$ contains all elements in $X$. For each $n$, the fact that $A$ is a witness of $\DS$ guarantees that there is some $Y$ such that $\HJ^{X_{n + 1}} = Y$ and $Y \leq_H X_n$. Since $\M$ contains $X_n$ and since $\beta$-models are closed under hyperarithmetic reducibility, $\M$ contains $Y$. And by proposition \ref{absoluteness}, $\beta$-models are sufficiently correct that $\M \vDash ``Y = \HJ^{X_{n + 1}} \land Y \leq_H X_n."$

Suppose $X$ is a sequence such that $\varphi(M, X)$ holds. Then for each $n$ there is a $Y$ such that $\M \vDash ``Y = \HJ^{X_{n + 1}} \land Y \leq_H X_n."$ By proposition $\ref{absoluteness}$, both clauses of the conjunction are absolute between $\beta$-models. Hence $\HJ^{X_{n + 1}}$ exists and is hyperarithmetical in $X_n$. So $X$ is a witness of $\DS$.

\begin{claim}
There is a model of $\ACA_0 + \DS$.
\end{claim}

By proposition \ref{hyperjump}, there is a $\beta$-model $\N$ that contains $M$. Since $\N$ is a $\beta$-model, by proposition \ref{beta}, it is a model of $\ACA_0$.

Since the $\Sigma^1_1$ formula $\exists X\, \varphi(M, X)$ holds and $\N$ is correct for $\Sigma^1_1$ formulas with parameters from $\N$, there is some $X$ in $\N$ such that $\N \vDash \varphi(M, X)$. And since $\N$ is a $\beta$-model, it is correct about this fact---that is, $\varphi(M, X)$ really does hold. Since $\varphi(M, X)$ holds, $X$ is a witness to $\DS$. The point now is just that $\N$ is correct enough to see that $X$ is a witness to $\DS$. In detail: for each $n$, $\HJ^{X_{n + 1}}$ exists and is hyperarithmetical in $X_n$. Since $X_n$ is in $\N$, this means $\HJ^{X_{n + 1}}$ is in $\N$. And by proposition \ref{absoluteness}, $\N$ agrees that it is the hyperjump of $X_{n + 1}$ and that it is hyperarithmetical in $X_n$. Therefore $\N$ agrees that $X$ is a witness to $\DS$.
\end{proof}

\begin{remark}
The previous proof actually demonstrates that $\ACA_0$ proves Theorem \ref{thm-main}. The original Spector proof relies on the theory of admissible ordinals, so it is unlikely to be formalizable in systems weaker than $\ATR_0$.
\end{remark}

\section{Semantic Incompleteness Theorems}
\label{incompleteness}

Steel derives the following theorem as a corollary of his Theorem \ref{Steel}.

\Steelincompleteness*

Because $\omega$-models are correct for arithmetic statements, we can restate this as

\begin{corollary}
Let $T$ be an arithmetically axiomatized $\Lang_2$ theory extending $\ACA_0$. If there is an $\omega$-model of $T$ then there is an $\omega$-model of
\[
T + \textrm{``there is no $\omega$-model of T''}.
\]
\end{corollary}

Similarly, we can use Theorem \ref{thm-main} to prove a stronger version of a theorem originally proved by Mummert and Simpson in \cite{mummert2004}. Note that in our version we do not need to assume that $T$ is recursively axiomatized, only that it has a $\beta$-model.

\begin{theorem}\label{minimalmodel}
Let $T$ be an $\Lang_2$ theory. If there is a $\beta$-model of $T$ then there is a $\beta$-model of $T$ that contains no countable coded $\beta$-models of $T$.
\end{theorem}

\begin{proof}
Suppose not. Then every $\beta$-model of $T$ contains a countable coded $\beta$-model of $T$. Let $\M$ be a $\beta$-model of $T$. So $\M$ contains some countable coded $\beta$-model $\N_0$ coded by a real $N_0$. Similarly $\N_0$ contains a countable coded $\beta$-model of $T$, $\N_1$, coded by a real $N_1$. In this manner we can define a sequence of countable $\beta$-models of $T$, $\N_0, \N_1, \N_2, \ldots$ along with their codes $N_0, N_1, N_2, \ldots$

But for each $n$, $N_{n + 1} \in \N_n$ and since $\N_n$ is a $\beta$-model it is correct about all $\Pi^1_1$ facts about $N_{n + 1}$. In other words, $\HJ^{N_{n + 1}}$ is arithmetic in $N_n$. So $N_0, N_1, \ldots$ provides an example of the type of descending sequence in the hyperdegrees shown not to exist in theorem \ref{thm-main}.
\end{proof}

In fact, this same proof actually yields a seemingly stronger result. A $\beta_n$-model is defined to be an $\omega$-model of second order arithmetic which is correct for all $\Sigma^1_n$ statements with parameters from the model. The same proof as above proves the theorem mentioned in the introduction (where once again our new proof shows that the assumption that $T$ is recursively axiomatized can be dropped):

\ms*

Since the statement that a real is the code for a $\beta_n$-model is $\Pi^1_n$, $\beta_n$-models are correct about such statements. And if $T$ is a $\Sigma^1_n$ axiomatized theory then $\beta_n$-models are also correct about which formulas are in $T$. Thus for $\Sigma^1_n$ axiomatized theories we can restate the above theorem to get the following theorem of Mummert and Simpson:

\begin{theorem}\label{ms2}
Let $T$ be a $\Sigma^1_n$ axiomatized $\Lang_2$ theory. If there is a $\beta_n$-model of $T$, then there is a $\beta_n$-model of $$\textrm{$T$+``there is no countable coded $\beta_n$-model of $T$.''}$$
\end{theorem}

\begin{proof}
Let $\M$ be a $\beta_n$-model of $T$ which contains no countable coded $\beta_n$-models of $T$. For any $N \in \M$, the statement that $N$ is not a countable coded $\beta_n$-model of $T$ is a true $\Sigma^1_n$ sentence and thus is satisfied by $\M$. Therefore $\M$ satisfies the statement ``there is no countable coded $\beta_n$-model of $T$.''
\end{proof}

From this we immediately infer the following corollary, a strengthened version of Mummert and Simpson's Corollary 2.4 from \cite{mummert2004}:

\begin{corollary}\label{ms3}
Let $T$ be a $\Sigma^1_n$ axiomatized $\Lang_2$ theory. If $T$ has a $\beta_n$-model then $T$ has a $\beta_n$ model that is not a $\beta_{n+1}$ model.
\end{corollary}

\begin{proof}
Let $T$ be a $\Sigma^1_n$ axiomatized $\Lang_2$ theory with a $\beta_n$ model. By Theorem \ref{ms2}, there is a $\beta_n$ model $\M$ of
$\textrm{$T$+``there is no countable coded $\beta_n$-model of $T$.''}$ The latter is a false $\Pi^1_{n+1}$ sentence, whence $\M$ is not a $\beta_{n+1}$ model.
\end{proof}

\begin{remark}
The definability restriction on $T$ in Theorem \ref{ms2} and Corollary \ref{ms3} can be relaxed slightly to the requirement that $T$ have an axiomatization definable by a formula consisting of a series of first order quantifiers over a Boolean combination of $\Sigma^1_n$-formulas. The proof is essentially the same.
\end{remark}

\subsection{Further discussion of definability restrictions}
When this paper was first published, Theorem \ref{ms2} and Corollary \ref{ms3} were missing the definability restriction on $T$. We will now address whether such a restriction is necessary. Before doing so, it will be helpful to state a result implicit in the proof of Theorem \ref{minimalmodel}, since we will use it several times below.

\begin{theorem}\label{minimalmodel2}
The relation ``contains a countable coded model isomorphic to'' is a well-order on $\beta$-models.
\end{theorem}

If the definability restriction is removed completely from Theorem \ref{ms2}, it is not even clear what the statement means. The most reasonable interpretation is to assume that $T$ is definable in second order arithmetic (but with no bound on the complexity of the definition) and to use the formula defining $T$ to write ``there is no countable coded $\beta_n$-model of $T$'' as an $\Lang_2$ sentence. Unfortunately, this interpretation is false, as the following proposition demonstrates.

\begin{proposition}
There is an $\Lang_2$ theory $T$ and a $\Sigma^1_2$ formula $\psi$ which defines $T$ (relative to some fixed, standard coding of $\Lang_2$ sentences) such that $T$ has a $\beta$-model but every $\beta$-model of $T$ satisfies
\[
\text{``there is a countable coded $\beta$-model of the theory defined by $\psi$.''}
\]
\end{proposition}

\begin{proof}
Let $\varphi_1$ be the sentence ``there is a countable coded $\beta$-model'' and let $\varphi_2$ be the sentence ``there is a countable coded $\beta$-model which contains a countable coded $\beta$-model.'' Note that $\varphi_2$ is a true $\Sigma^1_2$-sentence. Let $T$ be the theory $\{\varphi_1\land\lnot\varphi_2\}$. By Theorem \ref{minimalmodel2}, there is a $\beta$-model of $T$.

Now assume that $\langle\theta_n\rangle_{n \in \mathbb{N}}$ is a fixed, standard enumeration of all $\Lang_2$-sentences. Let $\psi(n)$ be the formula
\[
\varphi_2\land (\text{$\theta_n =$ ``$\varphi_1\land\lnot\varphi_2$''}).
\]
In other words, if $\varphi_2$ holds then $\psi$ defines $T$ and if $\varphi_2$ does not hold then $\psi$ defines the empty theory. Since $\varphi_2$ is true, $\psi$ is a $\Sigma^1_2$ definition of $T$.

Now let $\M$ be any $\beta$-model of $T$. Since $\M \models \varphi_1$, $\M$ believes there is a countable coded $\beta$-model. And since $\M \models \lnot\varphi_2$, $\M$ believes that $\psi$ defines the empty theory and thus that the theory defined by $\psi$ is satisfied in every model. Thus $\M$ believes there is a countable coded $\beta$-model of the theory defined by $\psi$.
\end{proof}

Thus the definability restriction on Theorem \ref{ms2} cannot even be relaxed to include all $\Sigma^1_{n + 1}$ axiomatized theories. The situation for Corollary \ref{ms3} is less clear. In particular, the authors do not currently know if the result holds when the definability restriction is removed, and consider this to be an interesting question. The following two results demonstrate why finding a counterexample may not be easy.

\begin{proposition}\label{separation}
Let $T$ be an $\Lang_2$ theory. If $T$ has a $\beta_n$-model that contains $T$ then $T$ has a $\beta_n$-model that is not a $\beta_{n + 1}$-model.
\end{proposition}

\begin{proof}
By Theorem \ref{minimalmodel2}, we can find a $\beta_n$-model $\M$ of $T$ which contains $T$ but which contains no countable coded $\beta_n$-model of $T$ containing $T$. Hence the statement
\[
\text{``there is a countable coded $\beta_n$-model of $T$''}
\]
does not hold in $\M$. Since it is a true $\Sigma^1_{n + 1}$ statement with parameters from $\M$ (since $\M$ contains $T$), $\M$ is not a $\beta_{n + 1}$-model.
\end{proof}

\begin{corollary}
Let $T$ be a true $\Lang_2$ theory. Then for all $n \ge 1$, $T$ has a $\beta_n$-model that is not a $\beta_{n + 1}$-model.
\end{corollary}

\begin{proof}
If $T$ is a true theory, then for every $n$ there is a $\beta_n$-model containing $T$, so we may apply Proposition \ref{separation}.
\end{proof}

\section{Spector Ranks}
\label{ranks}

Define a relation $\prec$ on pairs of reals by $A \prec B$ iff $\HJ^A \leq_H B$. By theorem \ref{thm-main}, this relation is well-founded and therefore reals can be assigned ordinal ranks according to it. Let's refer to the $\prec$-rank of a real as its \emph{Spector rank}. In this section we will calculate the Spector ranks of reals, showing that we get the same ranks as those induced by the $\omega_1$'s of reals.

We will need to use the following theorems:

\begin{theorem}[Spector]
\label{spector}
For any reals $A$ and $B$:
\begin{enumerate}
    \item If $\HJ^B \leq_H A$ then $\omega_1^B < \omega_1^A.$
    \item If $B \leq_H A$ and $\omega_1^B < \omega_1^A$ then $\HJ^B \leq_H A.$
\end{enumerate}
\end{theorem}

\begin{theorem}[Sacks]
\label{sacks}
If $\lambda$ is an admissible ordinal greater than $\omega$ and $X$ is a real such that $X$ computes a presentation of $\lambda$ (i.e.\ $\lambda < \omega_1^X$) then there is a real $Y$ that is hyperarithmetical in $X$ such that $\omega_1^Y = \lambda$.
\end{theorem}

\begin{remark}
  Theorem \ref{sacks} is typically stated without the requirement that $Y$ is hyperarithmetical in $X$, though this is implicit in all or nearly all extant proofs of the theorem. For instance, in \cite{steel1978} Steel uses the method of forcing with tagged trees to prove Sacks' theorem. In that case, the real $Y$ is obtained as the reduct of a generic filter over $L_\lambda$. Since any presentation of $\lambda$ can hypercompute such a generic (if you can compute a presentation of $\lambda$ then it just takes $\omega\cdot(\lambda + 1)$ jumps to compute the theory of $L_\lambda$), $X$ can hypercompute a $Y$ witnessing Sacks' theorem.
\end{remark}

The calculation of Spector ranks now follows relatively easily.

\ranks*

\begin{remark}
The only reason we need to say $(1 + \alpha)^\text{th}$ admissible rather than $\alpha^\text{th}$ admissible is that the way admissible is usually defined, $\omega$ is an admissible ordinal but unlike all other countable admissible ordinals, it is not the $\omega_1$ of any real.
\end{remark}

\begin{proof}
We will argue by induction on $\alpha$ that for any $A$ if $\rank(A) > \alpha$ then $\omega_1^A$ is greater than the $(1 + \alpha)^\text{th}$ admissible ordinal and conversely that if $\omega_1^A$ is greater than the $(1 + \alpha)^\text{th}$ admissible then $\rank(A) > \alpha$.

First suppose $\rank(A) > \alpha$. So there is some $B$ of rank $\alpha$ such that $\HJ^B \leq_H A$. By Spector's result, theorem \ref{spector}, this implies $\omega_1^B < \omega_1^A$. And by the induction assumption, $\omega_1^B$ is at least the $(1 + \alpha)^\text{th}$ admissible so $\omega_1^A$ is greater than the $(1 + \alpha)^\text{th}$ admissible.

Now suppose that $\omega_1^A$ is greater than the $(1 + \alpha)^\text{th}$ admissible. Let $\lambda$ denote the $(1 + \alpha)^\text{th}$ admissible. By Sacks' theorem, there is some $B$ hyperarithmetical in $A$ such that $\omega_1^B = \lambda$. Since $\omega_1^B < \omega_1^A$, Spector's theorem implies that $\HJ^B \leq_H A$ and hence $\rank(B) < \rank(A)$. By the induction assumption, $\rank(B)$ is at least $\alpha$, so $\rank(A) > \alpha$.
\end{proof}

\begin{theorem}[Silver]
If $\alpha$ is admissible relative to $0^\sharp$ then $\alpha$ is a cardinal in $L$.
\end{theorem}

Hence if $X$ is a real in the cone above $0^\sharp$ then $\omega_1^X$ is a cardinal in $L$. Suppose that $\omega_1^X$ is the $\alpha^\text{th}$ admissible. Since $\omega_1^X$ is a cardinal in $L$, it follows that actually $\alpha = \omega_1^X = \omega_\alpha^{CK}$. So if $0^\sharp$ exists then on a cone, the Spector rank of a real $X$ is equal to $\omega_1^X$.

\begin{theorem}
If $0^\sharp$ exists, then for all $A$ on a cone, the Spector rank of $A$ is $\omega_1^A$.
\end{theorem}

Alternatively, one can infer the previous theorem from the following proposition due to Martin.

\begin{proposition}[Martin]
Assuming appropriate determinacy hypotheses, if $F$ is a degree invariant function from reals to
(presentations of) ordinals such that $F(A)\leq \omega_1^A$, then either
$F$ is constant on a cone or $F(A)=\omega_1^A$ on a cone.
\end{proposition}

One could also consider the analogous relation given by replacing hyperarithmetic reducibility and the hyperjump with Turing reducibility and the Turing jump. Namely, define $\prec_T$ by $A \prec_T B$ iff $A' \leq_T B$. By results of Harrison (see \cite{harrison1968}), this relation is not well-founded. However, it \emph{is} well-founded if we restrict ourselves to the hyperarithmetic reals, as shown by Putnam and Enderton in \cite{enderton1970}. In that paper, Putnam and Enderton also show that the rank of a hyperarithmetic real $A$ in this relation is ``within $2$'' of the least $\alpha$ such that $A$ cannot compute $0^{(\alpha)}$. More precisely, if the rank of $A$ is $\alpha$ then $A$ cannot compute $0^{(\alpha + 1)}$ and if $A$ cannot compute $0^{(\alpha)}$ then the rank of $A$ is at most $\alpha + 2$.

\bibliographystyle{plain}
\bibliography{hyperjump}

\end{document}